\documentclass{siamltex}

\usepackage{appendix}
\usepackage{amsmath}
\usepackage{psfrag}
\usepackage{amssymb}
\usepackage{comment}          
\usepackage{float}
\usepackage[english]{babel}   
\usepackage{array}
\usepackage{enumerate}
\usepackage{calc}
\usepackage{xspace}
\usepackage{graphicx,color,subfigure}

\newcommand{\SCAL}{{\cdot}}
\newcommand\eqdef{:=}

\newcommand\T{\mathcal{T}}
\newcommand\F{\mathcal{F}}

\newcommand{\jump}[1]{[\![#1]\!]}

\newcommand{\tnorm}[1]{\vert\hspace{-0.3mm}\Vert#1\Vert\hspace{-0.3mm}\vert}

\newcommand{\normal}{n}

\newcommand{\tr}{^{\rm T}}

\title{Robust error estimates for stabilized finite element
  approximations of the two dimensional
  Navier-Stokes equations with application to
  implicit large eddy simulation}

\author{Erik Burman\thanks{Department of Mathematics, 
University College London, Gower Street, London, 
UK--WC1E  6BT, 
United Kingdom; ({\tt e.Burman@ucl.ac.uk})}
}
\begin{document}

\maketitle

\begin{abstract}
We consider error estimates in weak parametrised norms for stabilized finite element
approximations of the two-dimensional Navier-Stokes' equations. These
weak norms can be related to the norms of certain filtered quantities,
where the parameter of the norm, relates to the filter width. Under
the assumption of the existence of a certain decomposition
of the solution, into large eddies and fine scale fluctuations, the constants of the
estimates are proven to be independent of both the Reynolds number and the
Sobolev norm of the exact solution. Instead they exhibit exponential
growth with a coefficient proportional to the maximum gradient of the
large eddies. The error estimates are on a posteriori form, but using
Sobolev injections valid on finite element spaces and the properties
of the stabilization operators the residuals may be
upper bounded uniformly, leading to robust a priori error estimates.
 \end{abstract}



\pagestyle{myheadings}
\thispagestyle{plain}
\markboth{E. BURMAN}{
Robust error estimates for the two dimensional
  Navier-Stokes equations}

\section{Introduction}
In this paper we will be interested in stabilized finite element
methods in the context of so called implicit large eddy
simulation (ILES), see \cite{Bo90}. This is a numerical approach to
the computation of turbulent flow where no modelling of the Reynolds
stresses is performed on the continuous level. Instead the Navier-Stokes' equations are approximated
numerically using a method that dissipates sufficient energy on the scale of the
mesh size. This eliminates the buildup of energy that creates
spurious oscillations in any energy conservative approximation
method. It has been argued that the truncation error of such methods
by itself may 
act as a subgrid model \cite{AND08, MRG06} and there exists numerical evidence that ILES methods work for the
simulation of two dimensional turbulence, provided
back scatter effects are not strong \cite{KTW12}. There is also
numerical evidence of the potential for adaptive LES/DNS
driven by adaptive, stabilized finite element simulations, see
\cite{Hoff05,HJ06} and \cite{PCH10}.

Our objective in this paper
is to provide a numerical analysis for stabilized finite element
methods under minimal regularity assumptions and to provide sufficient
conditions on the exact solution for the derivation of rigorous error
estimates that are independent of both the Reynolds number and Sobolev
norms of the exact solution. It is well
known that provided the exact solution is sufficiently smooth the
approximate solution $u_h$ of the Navier-Stokes' equations on velocity-pressure form
can be proved to satisfy
estimates
of the type
\begin{equation}\label{classic_estimate}
\|u - u_h\|_{L^2(\Omega)} \lesssim
e^{\|\nabla u\|_\infty} h^{\frac32}
  |u|_{L^2(I;H^2(\Omega))},
\end{equation}
if a consistent stabilized finite element method with piecewise affine
approximation is used.
Here $u$ denotes the flow velocity and $h$ the mesh size.
See \cite{HS90, BF07} for examples of analyses of Navier-Stokes' equations on
velocity-pressure form and \cite{JS86, LS00} for analyses on
velocity-vorticity form. We also give a proof of
\eqref{classic_estimate} for one the methods proposed herein in appendix.
Here we use the notation $a \lesssim b$ for $a \leq C b$ with $C$ a
constant independent of the physical parameters of the problem, unless
they can be expected to have $O(1)$ contribution, it can also include
some dependence on initial data, that may be assumed to be $O(1)$. We
will also use $a \sim b$ for $a \lesssim b$ and $b \lesssim a$.
Note that there is no explicit dependence on the viscosity in the
estimate
\eqref{classic_estimate}.
For this estimate to be useful the included Sobolev norms must be small, which
rarely is the case in the high Reynolds number regime and hence the 
dependence of the viscosity enters in an implicit manner. The purpose
of the present paper is to propose an alternative approach, where the
estimate is indeed independent of the Reynolds number, both in the
sense that the estimate is free from inverse power of the viscosity in
the upper bound, but also that the dependence on unknown Sobolev norms of the exact
solution is strongly reduced. It does not seem possible to eliminate
this dependence completely, due to the possible presence of
backscatter. 
Observe that the presence of $\|\nabla u\|_{\infty}$ in the
exponential of \eqref{classic_estimate} reflects the effect of diverging
characteristics in the transport equation and is present already in
the linear convection--diffusion equation at high P\'eclet numbers. 
We may define a timescale
for the flow separation, $\tau:=\|\nabla
u\|_{\infty}^{-1}$. The reason this time scale becomes so small
is that it will be the smallest timescale of the flow and hence equal to the micro time, because of fine scale
fluctuations of the velocity. From the physical point of view it is argued that \emph{LES will be
successful
for flows where both the quantities of interest and the
rate-controlling processes are determined by the resolved large
scales}, see Pope \cite{Pope04} for a discussion. We will use this
idea as a starting point for our assumptions on the flow.

To derive error estimates for a numerical method we need the following:
\begin{itemize}
\item[--] continuous dependence on data, independent of the exact solution;
\item[--] some smooth quantity that we can apply approximation estimates to.
\end{itemize}
At a first glance both these prerequisites appear to fail for the
two-dimensional Navier-Stokes' equation. The first fails because of the
presence of the exponential factor and the second fails because
Sobolev norms of the exact solution can be huge for small viscosities.
The following three points allow us to break this deadlock:
\begin{enumerate}
\item the use of a parametrized weak norm, corresponding to measuring
  the error in filtered quantities of the solution;
\item introduction of an assumption on the structure of the exact
  solution that is sufficient for an implicit large eddy simulation
  to be robust;
\item a stabilized finite element method, giving enhanced a priori
  control of residual quantities in the high Reynolds regime.
\end{enumerate}
The idea of measuring error in filtered quantities was considered in
\cite{DJ04,DJL04}, but the estimates were not robust in the Reynolds
number and the constant included high order Sobolev norms of the exact solution.
In \cite{Bu12} weak norm estimates were used in order
to derive robust estimates for the Burgers' equation, where the
constant in the right hand side only depended on initial data.
The second point, which was not necessary in the case of the Burgers'
equation, reflects the difficulty to characterise the solution
structure in higher dimension and the ensuing need of some structural
assumptions in order to rule out strong backscatter effects. The third
point allows us to control residual quantities independent of the viscosity.

As a first approximation it is reasonable to assume that
for a solution to be amenable to large eddy simulation, there are 
relatively smooth eddies, with large associated Reynolds number, 
containing the bulk of the energy and small
scale fluctuations that may vary rapidly in space, but carries a
negligible part of the energy. To make this precise, we assume that
there exists a decomposition of the exact solution in the spatially
slowly varying large scales and an arbitrarily rough fine scale, with
small energy. 
\[
u = \bar u + u',\quad \bar u, u' \in W^{1,\infty}(\Omega).
\]
We then assume that the Reynolds number associated to the large
scales, $\overline{Re}$
may be large, but $\|\bar u\|_{W^{1,\infty}(Q)} \sim 1$, whereas for
the fine scales $\|u'\|_{W^{1,\infty}(Q)}$ may be large, but the
energy small. To give a precise meaning to small here, we 
introduce a global time scale for the flow, defined using the large
scales
\[
\tau_F := \|\bar u\|_{W^{1,\infty}(Q)}^{-1} \sim 1.
\]
This is in agreement with the statement that rate controlling
processes are determined by the large scale.
Using the viscosity coefficient we may make the following assumption
on the energy content of the small scales
\[
\|u'\|^2_{\infty}  \sim \nu/\tau_F.
\]
The length scale based on $|u'|$ and $\tau_F$ writes
\[
l':=|u'| \tau_F \sim |u'| \frac{\nu}{|u'|^2}
\]
and it
follows that the small scale Reynolds number is
\begin{equation}
Re' := \frac{u' l'}{\nu} \sim \frac{|u'| \nu}{|u'| \nu} \sim 1.
\end{equation}
Alternatively one may assume that the fine scale Reynolds number is
one and that the large scale characteristic time, is the globally relevant
time scale and then derive the bound on the energy. We will refer to
the above as the \emph{large eddy assumption}. Under this assumption
we prove the following bound on the approximate velocities
\begin{equation}\label{ILES_estimate}
\sup_{t \in (0,T)} \|u-u_h\|_{L^2(\Omega)} \lesssim h^{\frac12}.
\end{equation}
The hidden constant in the above estimate only depends on initial data
(maximum initial vorticity) and the mesh geometry, but is
independent of the Reynolds number and Sobolev norms of the exact
solution. The discussion is limited to two space dimensions and hence 
we do not properly speaking address the question of turbulent flows.

Let us end this introductory discussion by emphasising that what we compute is an
approximation to the solution of the Navier-Stokes' equations. For
this approximate solution we can
prove that provided $\tau_F$ is not too small, corresponding to slowly
varying large scale velocity field, the filtered part of the vorticity
is stable under perturbations resulting in robust error estimates in
weak norms for vorticity. Using these estimates we may then control the
$L^2$-norm of the velocity error as shown above. Herein
our main concern will be the high mesh Reynolds number case
\[
Re_h := \frac{U_0 h}{\nu} >1,
\]
where $U_0:=\|u_h(\cdot,0)\|_{L^\infty(\Omega)} \sim 1$ denotes a
characteristic velocity of the flow, but many results are independent
of the mesh Reynolds number. It will always be explicitly stated when a
result only holds in the high Reynolds regime. If the local Reynolds
number is low, other approaches than those presented herein might be
more appropriate. Let us also point out that another feature of our
estimates is that they provide the first error estimates with an order
in $h$ for  nonlinear stabilization schemes, satisfying a discrete
maximum principle, in two space dimensions.

We will consider the Navier-Stokes' equations written on
vorticity-velocity form. Let $\Omega$ be the unit square and assume
that the boundary conditions are periodic in both cartesian
directions. The $L^2$-scalar product over some
space-time domain will be denoted $(\cdot,\cdot)_X$ with associated
norm $\|\cdot\|_X$ where the subscript may be dropped for $X=\Omega$. 
Define the time interval $I:=(0,T)$ and the space-time domain $Q:= \Omega \times I$.
The equations then writes,
\begin{align}\nonumber
\partial_t \omega + \nabla \SCAL (u \omega) - \nu \Delta \omega &= 0,
\mbox{ in } Q,\\ \label{NSeq}
-\Delta \Psi & = \omega \mbox{ in } Q,\\ \nonumber
u & = \mbox{rot } \Psi \mbox{ in } Q,\\ \nonumber
\omega(x,0) &= \omega_0,
\end{align}
with $\omega_0 \in L^\infty(\Omega)$.
The associated weak formulation takes the form, for $t>0$, find $(\omega,\Psi)
\in H^1(\Omega) \times H^1(\Omega)\cap L_*(\Omega)$ such that
\begin{align}\label{NSweak1}
(\partial_t \omega,v)_\Omega + a(u;\omega,v)&= 0,\\ \label{NSweak2}
(\nabla\Psi,\nabla \Phi)_\Omega & = (\omega,\Phi)_\Omega,\\
u & = \mbox{rot } \Psi, \nonumber
\end{align}
for all $(v,\Phi) \in H^1(\Omega) \times H^1(\Omega)\cap
L_*(\Omega)$, where the semi-linear form $a(\cdot;\cdot,\cdot)$ is
defined by
\[
a(u;\omega,v):=(\nabla \SCAL (u \omega),v)_\Omega +
(\nu \nabla \omega, \nabla v)_\Omega.
\]
This problem is known to be well-posed, but a priori error estimates
on the solution are in general strongly dependent on the viscosity
coefficient reflecting the possible poor stability of the equations in
the high Reynolds number regime. 
\section{Finite element discretization}
\label{sec:disc.set}

Let $\{\T_h\}_{h>0}$ be a family of affine, simplicial Delaunay meshes of $\Omega$. 
We assume that the meshes are kept fixed
in time and that the family $\{\T_h\}_{h>0}$ is
quasi-uniform. Mesh faces are collected in the set $\F$. For a smooth enough function $v$ that
is possibly double-valued at $F\in\F$ with $F=\partial T^- \cap
\partial T^+$, we define its jump at
$F$ as $\jump{v}\eqdef v|_{T^-}-v|_{T^+}$, and we fix the unit normal
vector to $F$, denoted by $\normal_F$, 
as pointing from $T^-$ to $T^+$. The
arbitrariness in the sign of $\jump{v}$ is irrelevant in what
follows. Define $V_h$ to be the standard space of piecewise polynomial,
continuous periodic functions,
\[
V^k_h\eqdef \{v_h \in H^1(\Omega): v_h\vert_{K} \in P_k(K); \forall K
\in \mathcal{T}_h; 
\, v_h \mbox{
  periodic in $x$ and $y$} \}.
\]
The set of gradients of functions in $V^k_h$ will be denoted by
$$W^{k-1}_h:=\{w_h = \nabla v_h;\, v_h \in V_h^k\}.$$
Let $L\eqdef L^2(\Omega)$ and set $L_* \eqdef \{q\in L; \; \int_\Omega
q=0\}$. Let $V_*\eqdef V^k_h \cap L_*$. We let $\pi_L$ denote the
$L^2$-projection on $V_h^k$ and $\pi_V$ the $H^1$-projection
\[
(\nabla \pi_V u, \nabla v_h)_\Omega = (\nabla u, \nabla v_h)_\Omega
\quad \forall v_h \in V_h \mbox{ and }
\int_\Omega (\pi_Vu - u) ~\mbox{d}x = 0.
\]
We recall that the following approximation estimates hold for $\pi_L$
and $\pi_V$,
\begin{equation}\label{approx1}
\|\pi_L u - u_h\| + h \|\nabla (\pi_L u - u)\| \leq c_0 h^{s}
|u|_s, \mbox{ with } 1 \leq s \leq k+1 
\end{equation}
and
\begin{equation}\label{approx2}
\|\pi_V u - u_h\| + h \|\nabla (\pi_V u - u)\| \leq  c_1 h^{s}
|u|_s, \mbox{ with } 1 \leq s \leq k+1. 
\end{equation}
We consider continuous finite elements with $k=1$ to
discretize  the vorticity $\omega$ in space and $k=1,2$ for the stream function
$\Psi$. The discrete velocity is given elementwise by $u_h\vert_K := \mbox{rot }
\Psi_h:= (\partial_y \Psi_h, - \partial_x \Psi_h)$. Note that using this definition $\nabla \cdot u_h = 0$ in
$\Omega$, i.e. the discrete velocity is globally divergence free. We
discretise in space using a stabilized finite element method. For
$t>0$ find $(\omega_h,\Psi_h) \in V^1_h \times V^l_*$ , with $l=1,2$, such that
\begin{align}\label{NSFEM1}
(\partial_t \omega_h,v_h)_M + a(u_h;\omega_h,v_h) +s(u_h;\omega_h,v_h) &= 0
,\\
(\nabla\Psi_h,\nabla \Phi_h)_\Omega - (\omega_h,\Phi_h)_\Omega&=0, \label{NSFEM2}\\
u_h - \mbox{rot } \Psi_h & = 0,\nonumber
\end{align}
for all $(v_h,\Phi_h) \in V_h \times V_*$ and with initial data $w_0
:= \pi_L \omega(\cdot,0)$. Here $s(\cdot;\cdot,\cdot)$
denotes a stabilization operator that is linear in its last
argument and $(\partial_t \omega_h,v_h)_M$ denotes the bilinear form
defining the mass matrix, this operator either coincides with
$(\cdot,\cdot)_\Omega$ or is defined as $(\cdot,\cdot)_\Omega$ approximated
using nodal quadrature, i.e. so called mass lumping. We will assume the stabilization term satisfies
\begin{equation}\label{stab_bound1}
\inf_{v_h \in V_h} \|h^{\frac12}(u_h\cdot \nabla \omega_h - v_h)\| \lesssim
s(u_h,\omega_h;\omega_h)^{\frac12} \lesssim h^{\frac12}
(U_0+\|u_h\|_{L^\infty(\Omega)}) \|\nabla \omega_h\|
\end{equation}
and
\begin{equation}\label{stab_bound2}
s(u_h,\omega_h;v_h) \lesssim h^\frac12(U_0+\|u_h\|_{L^\infty(\Omega)})
s(u_h,\omega_h;\omega_h)^\frac12 \|\nabla v_h\|.
\end{equation}
The formulation \eqref{NSFEM1}-\eqref{NSFEM2} satisfies the following stability
estimates
\begin{lemma}\label{FEMstability}
\begin{equation}\label{enstrophy_stability}
\sup_{t \in I} \|\omega_h(\cdot,t)\|_M^2 + 2\|\nu^{\frac12} \nabla
\omega_h\|^2_{Q} + 2\int_I s(u_h;\omega_h,\omega_h)~\mbox{d}t \leq \|\omega_h(\cdot,0)\|_M^2,
\end{equation}
and if exact integration is used for $(\cdot,\cdot)_M$, 
\begin{equation}\label{energy_stability}
\|u_h(\cdot,T)\|_M^2 + 2\|\nu^{\frac12}
\omega_h\|^2_{Q} = \|u_h(\cdot,0)\|_M^2-2\int_0^T s(u_h; \omega_h, \Psi_h) ~\mbox{d}t,
\end{equation}
\begin{equation}\label{ubound_infty}
\|u_h(\cdot,t)\|_{L^{\infty}(\Omega)} \leq c_q
\|\omega_h(\cdot,t)\|_{L^q(\Omega)},\quad q>2
\end{equation}
and for $l=1$, 
\begin{equation}\label{timederstab}
\int_0^T\|\nabla \partial_t \omega_h\| \mbox{d}t \lesssim 
\int_0^T (h^{-\frac32} (U_0 +\|u_h\|_{L^{\infty}(\Omega)})
  s(u_h,\omega_h,\omega_h)^{\frac12} + \nu h^{-2} \|\nabla \omega_h\| ) \mbox{d}t.
\end{equation}
\end{lemma}
\begin{proof}
Inequality \eqref{enstrophy_stability} is immediate by taking $v_h =
\omega_h$ in \eqref{NSFEM1}. Inequality \eqref{energy_stability} is
obtained
by taking $v_h = \Psi_h$ in the equation \eqref{NSFEM1} and
deriving the equation \eqref{NSFEM2} in time and taking
 $\Phi_h = \omega_h$. For the inequality \eqref{ubound_infty},
 consider the auxiliary problem, $-\Delta \tilde \Psi = \omega_h$ in $\Omega$ and
 note that by \cite{RS82} there holds
\[
\|u_h(\cdot,t)\|_{L^{\infty}(\Omega)} \lesssim \|\tilde \Psi(\cdot,t)\|_{W^{1,\infty}(\Omega)}
\]
and adapting the analysis of \cite{Maz09} we have for the (simpler) case or periodic boundary conditions,
\[
\|\tilde \Psi(\cdot,t)\|_{W^{1,\infty}(\Omega)} \leq c_q
\|\omega_h(\cdot,t)\|_{L^q(\Omega)},\, q>2.
\]
To prove \eqref{timederstab} finally we introduce a function $\xi_h
\in V_h$ such that
\[
(\xi_h,v_h)_M = (\nabla \partial_t \omega_h, \nabla v_h)_\Omega,\quad
\forall v_h \in V_h,
\]
it follows by taking $v_h=\xi_h$ and using the Cauchy-Schwarz inequality
followed by an inverse inequality that
\begin{equation}\label{inv_xi}
(\xi_h,\xi_h)_M^{\frac12} \sim \|\xi_h\| \lesssim h^{-1} \|\partial_t \nabla \omega_h\|.
\end{equation}
Observe that by norm equivalence on discrete spaces the $L^2$-norm
defined using nodal quadrature is equivalent to the consistent $L^2$-norm.
Taking $v_h = \xi_h$ in \eqref{NSFEM1} yields
\[
\|\partial_t \nabla \omega_h\|^2 = -(u_h \cdot \nabla \omega_h,
\xi_h)_\Omega-(\nu \nabla \omega_h, \nabla \xi_h)_\Omega- s(u_h,\omega_h,\xi_h).
\]
We may then apply the Cauchy-Schwarz inequality in the second term
of the right hand side and \eqref{stab_bound2} in the last term,
followed by inverse inequalities on $\|\nabla \xi_h\|$ and the
estimate \eqref{inv_xi}. For
the first term we write, using the properties of $\xi_h$ and the bound 
$$
|(v_h,\xi_h)_M - (v_h,\xi_h)_\Omega| \lesssim (h^2 |\nabla v_h|,|\nabla \xi_h|)_\Omega,
$$
\[
|(u_h \cdot \nabla \omega_h,
\xi_h)_\Omega,| \lesssim |(u_h \cdot \nabla \omega_h - v_h,
\xi_h)_\Omega| + |(\partial_t \nabla \omega_h, \nabla v_h)_\Omega| + (h^2 |\nabla
v_h|,|\nabla \xi_h|)_\Omega
\]
Since both $u_h$ and $\nabla \omega_h$ are constant per
element $\nabla v_h\vert_K = \nabla (v_h - u_h \cdot \nabla
\omega_h)\vert_K$. Using inverse inequalities and the bound \eqref{inv_xi} on $\xi_h$ we
have
\[
 |(u_h \cdot \nabla \omega_h - v_h,
\xi_h)|+ |(\partial_t \nabla \omega_h, \nabla v_h)| + (h^2 |\nabla
v_h|,|\nabla \xi_h|) \lesssim h^{-1} \|\partial_t \nabla \omega_h\|
\|u_h \cdot \nabla \omega_h - v_h\|.
\]
The claim follows by the inequality \eqref{stab_bound1}, \eqref{inv_xi} and finally by integrating in time.
\end{proof}\\
It follows from \eqref{energy_stability} that the method is energy consistent if
$s(u_h;\omega_h,\Psi_h)=0$.
Taking the difference of the formulations \eqref{NSweak1} -
\eqref{NSweak2}  (with
$v=v_h$) and \eqref{NSFEM1} - \eqref{NSFEM2} and
setting $e_\omega = \omega-\omega_h$ and $e_{\Psi} = \Psi-\Psi_h$, the following consistency relation holds
\begin{align}\nonumber
(\partial_t e_\omega+u\SCAL\nabla e_\omega+\mbox{rot } e_{\Psi} \SCAL
\nabla \omega_h,v_h)_\Omega + (\nu \nabla e_\omega,\nabla v_h)_\Omega &= (\partial_t
\omega_h,v_h)_M - (\partial_t
\omega_h,v_h)_\Omega\\& \quad + s(u_h,\omega_h;v_h) \label{NSpert1}
\mbox{ in } Q,\\ 
(\nabla e_{\Psi},\nabla \Phi_h)_\Omega - (e_\omega,\Phi_h)_\Omega & = 0\mbox{ in } Q.\nonumber
\end{align}
As mentioned in the introduction, if the solution $(u,\omega)$ is
smooth one may prove
an error estimate that is robust with respect to $\nu$ using standard
linear theory and perturbation arguments. For the methods we consider
herein, this result is
an extension of the works in \cite{LS00} and
\cite{BF07} and we state it here only with the
dominant terms present. For the readers convenience we briefly outline
the proof using one stabilization operator (defined in  equation \eqref{s_cd}) in the appendix.
\begin{proposition}\label{smooth_apriori}
Let $(u,\omega)$ be a smooth solution of \eqref{NSweak1}-\eqref{NSweak2} and
$(u_h,\omega_h)$ be the solution of \eqref{NSFEM1}-\eqref{NSFEM2},
where the stabilization operator satisfies the additional weak
consistency property
\[
s(u_h;\pi_L \omega,\pi_L \omega) \leq c(u,\omega) h^{\frac32}
\]
then for $l=1,2$
\[
\|(u - u_h)(\cdot,T)\| + \|(\omega - \omega_h)(\cdot,T)\| \lesssim c_\omega (h^{\frac32} |\omega|_{L^2(I;H^2(\Omega))} + h^l |\Psi|_{L^\infty(I;H^{l+1}(\Omega))})
\] 
where $c_\omega := e^{\|\nabla
  \omega\|_{L^\infty(Q)} T}$.
In addition there holds for the stabilization operator
\[
 s(u_h;\omega_h,\omega_h)^{\frac12} \lesssim c_\omega (h^{\frac32} |\omega|_{L^2(I;H^2(\Omega))} + h^l |\Psi|_{L^\infty(I;H^{l+1}(\Omega))}).
\]
\end{proposition}
Observe that the exponential factor here depends on $\|\nabla
\omega\|_{L^\infty(\Omega)}$, compared to $\|\nabla
u\|_{L^\infty(\Omega)}$ in \eqref{classic_estimate}. This is the prize we pay for
estimating the $L^2$-error of the vorticity. As we shall see below,
the use of weaker norms for the estimation of $\omega_h$ allows us to
revert back to the exponential factor of \eqref{classic_estimate} and
under the large eddy assumption, the exponential growth is moderate.
\section{Dual problem}
From the consistency relation \eqref{NSpert1} we deduce the following (homogeneous)
perturbation formulation for the evolution of $(e_\omega,e_{\Psi})$
\begin{align}\label{NSpert2}
(\partial_t e_\omega+ u\SCAL\nabla e_\omega+\mbox{rot } e_{\Psi} \SCAL
\nabla \omega_h,\varphi_1)_\Omega + (\nu \nabla e_\omega,\nabla \varphi_1)_\Omega &= 0
\mbox{ in } Q,\\
(\nabla e_{\Psi},\nabla \varphi_2)_\Omega - (e_\omega,\varphi_2)_\Omega & = 0\mbox{
  in } Q,
\end{align}
where $\varphi_1, \varphi_2$ are the solutions to a dual adjoint perturbation
equation related to the continuous equation \eqref{NSweak1}-\eqref{NSweak2}  and the
discretization \eqref{NSFEM1}-\eqref{NSFEM2}. Since the jump of the
tangential derivative of $\omega_h$ is zero, we may integrate by parts
in \eqref{NSpert2}, to arrive at the dual adjoint problem
\begin{align}\label{eq:dual1}
-\partial_t \varphi_1 - u \SCAL \nabla \varphi_1 - \varphi_2 - \nu \Delta
\varphi_1 & = 0\mbox{ in } Q,\\
-\Delta \varphi_2 - \nabla \omega_h \SCAL \mbox{rot } \varphi_1 &=
0\mbox{ in } Q, \label{eq:dual2}\\
\varphi_1(x,T) & = \xi_0(x) \mbox{ in } \Omega, \label{eq:dual3}
\end{align}
where $\xi_0(x)$ is some initial data to be fixed later, the choice of
$\xi_0$ determines the quantity of interest. 

A key result for the present analysis is the following stability
estimate for the dual adjoint solution
\begin{proposition}\label{prop:dualstability}
The following stability estimate holds for the solution $(\varphi_1,\varphi_2)$ of
\eqref{eq:dual1} - \eqref{eq:dual3},
\begin{align}\label{phi_1stab}
\sup_{t \in I} \|\nabla \varphi_1(\cdot,t)\| +  \|\nu^{\frac12} D^2
\varphi_1\|_{Q} \lesssim e^{\frac{T}{\tau_F}} \|\nabla \xi_0\|\\
\label{phi_2stab}
\int_I \|\nabla \varphi_2(\cdot,t)\| ~\mbox{d}t
 \leq e^{\frac{T}{\tau_F}} \int_I  \|\omega_h\|_{L^\infty(\Omega)}~\mbox{d}t\,
 \|\nabla \xi_0\|
\end{align}
where $\tau_F$ is defined in the proof. If the large eddy assumption
holds $\tau_F\sim 1$.
\end{proposition}
\begin{proof}
First multiply \eqref{eq:dual1} by $-\Delta
\varphi_1$  and \eqref{eq:dual2} by $\varphi_1$ and integrate over $Q^*:=\Omega \times (t^*,T)$, where $t^*$ is
a time to be chosen. By summing the two relations we
obtain
\[
\underbrace{(\partial_t \varphi_1, \Delta \varphi_1)_{Q^*}}_{I_1} + \underbrace{ (u\SCAL \nabla
\varphi_1, \Delta \varphi_1)_{Q^*}}_{I_2} + \underbrace{ (\nabla \omega_h \SCAL \mbox{rot } \varphi_1, \varphi_1)_{Q^*}}_{I_3} + \underbrace{\|\nu^{\frac12} \Delta
\varphi_1\|^2_{Q^*}}_{I_4}= 0.
\]
We will now treat the terms $I_1$-$I_4$ term by term. First note that
by integration by parts first in space and then integration in time we
have
\[
I_1 = -\frac12 \int_{t^*}^T \frac{d}{dt} \|\nabla
\varphi_1(\cdot,t)\|^2 \mbox{d}t = \frac12 \|\nabla
\varphi_1(\cdot,t^*)\|^2 - \frac12 \|\nabla
\xi_0\|^2.
\]
The second term is handled using the decomposition of $u$ in the large
scale and fine scale component and then an integration by parts only in the
large scale part. Here $\nabla_S u$ denotes the symmetric part of the gradient of
the vector $u$.
\begin{multline*}
I_2 =  -((\nabla_S \bar u - \frac12 (\nabla \SCAL \bar u)
\mathcal{I}_{2\times2})  \nabla\varphi_1,\nabla \varphi_1)_{Q^*}
- (u' \SCAL \nabla \varphi_1,\Delta\varphi_1)_{Q^*} \\
\leq
 \int_Q (\Lambda(\bar
u,u',\nu) \nabla
\varphi_1)\tr \SCAL \nabla
\varphi_1 ~\mbox{d}x\mbox{d}t + \frac12 \|\nu^{\frac12} \Delta \varphi_1\|^2_{Q^*},
\end{multline*}
where $\Lambda(\bar
u,u',\nu)$ is a two by two, symmetric matrix defined by,
\[
\Lambda(\bar
u,u',\nu) = -\nabla_S \bar u + \frac12 \nabla \SCAL \bar u
\,\mathcal{I}_{2\times2}  + \frac{1}{2\nu} {u'}\tr u'.
\]
We now define the global timescale $\tau_F$ of the flow by
\[
(\tau_F)^{-1}:= \inf_{\bar u \in L^\infty(I;W^{1,\infty}(\Omega))} \|\sigma^+_p (\Lambda(\bar
u,u',\nu))\|_{L^{\infty}(Q)}.
\]
Here $\sigma^+_p$ denotes the largest positive eigenvalue of the
matrix. This results in a nontrivial minimization problem in
$L^\infty$. We leave the precise study of this problem for further
work and here simply observe that by computing the eigenvalues of the
symmetric part of the gradient tensor we may write
\[
(\tau_F)^{-1} \leq \inf_{\bar u}  J(\bar u,u')
\]
where 
\[
 J(\bar u,u') := \sup_{t \in I} \left(\sqrt{\|(\partial_{x_1} \bar u_1 - \partial_{x_2} \bar
  u_2
  )^2+(\partial_{x_2} \bar u_1 + \partial_{x_1} \bar u_2)^2\|_{L^\infty(\Omega)}} + \nu^{-1} \|u'\|^2_{L^\infty(\Omega)}\right).
\]
We observe that the global stability does not depend on the divergence
component or the rotational of $\bar u$, only on the other two components
of the velocity gradient matrix.
Since $u' = u - \bar u$, it follows that we can minimize over all large scale vector
fields $\bar u \in [W^{1,\infty}(\Omega)]$ and the infimum value
obtained is the optimal timescale of the flow. Under the assumptions
made in the introduction, that $\|\bar u\|_{W^{1,\infty}(\Omega)} \sim
1$ and $\nu^{-1} \|u'\|^2_{L^\infty(\Omega)} \sim 1$, for all $t$,
we immediately deduce that $\tau_F \sim 1$.

By an integration by parts and by using the relations
$\nabla \SCAL \mbox{rot } \varphi= 0$ and
$\nabla \varphi \SCAL \mbox{rot } \varphi = 0$ we have
\begin{equation*}
I_3 = -(\omega_h \nabla \SCAL \mbox{rot }
\varphi_1,\varphi_1)_{Q^*} -(\omega_h \mbox{rot } \varphi_1,\nabla
\varphi_1)_{Q^*} = 0.
\end{equation*}
Collecting the results for $I_1-I_3$ we have
\[
\|\nabla
\varphi_1(\cdot,t^*)\|_{L}^2 + \|\nu^{\frac12}\Delta
\varphi_1\|_{Q^*}^2 \leq \tau_F^{-1} \|\nabla
\varphi_1\|_{Q^*}^2+ \|\nabla
\xi_0\|_{L}^2.
\]
The inequality for $\varphi_1$ follows after a Gronwall's inequality
and by taking the supremum over $t^*$, resulting in
\[
\sup_{t \in I} \|\nabla
\varphi_1(\cdot,t)\|^2 + \|\nu^{\frac12}D^2
\varphi_1\|_{Q}^2 \lesssim e^{\frac{T}{\tau_F}} \|\nabla \xi\|^2.
\]
Elliptic regularity has been used for the second term.

For the bound on $\varphi_2$ multiply equation \eqref{eq:dual2} by
$\varphi_2$ and integrate over $\Omega$,
\[
\|\nabla \varphi_2(\cdot,t)\|^2 = -(\omega_h \mbox{rot } \varphi_1,\nabla
\varphi_2)_\Omega \leq \|\omega_h(\cdot,t)\|_{L^\infty(\Omega)} \|\nabla \varphi_1(\cdot,t)\|  \|\nabla \varphi_2(\cdot,t)\|.
\]
Then divide by $\|\nabla \varphi_2(\cdot,t)\|$, integrate in time
and use that
\[
\int_I \|\omega_h(\cdot,t)\|_{L^\infty(\Omega)} \|\nabla
\varphi_1(\cdot,t)\| ~\mbox{d}t \leq \int_I
\|\omega_h(\cdot,t)\|_{L^\infty(\Omega)} ~\mbox{d}t \,\sup_{t \in I} \|\nabla
\varphi_1(\cdot,t)\|.
\]
Finally use equation \eqref{phi_1stab} to bound the term in $\|\nabla
\varphi_1(\cdot,t)\|$.
\end{proof}\\
Note the dependence of $\omega_h$ in the bound \eqref{phi_2stab}. This
appearance of a finite element function in the stability estimate
shows that the global stability depends on the monotonicity of the
approximation scheme. However as we shall see, strict monotonicity is
not necessary, only $L^\infty$-control of the vorticity.
\section{A posteriori and a priori error estimates for the abstract method}
Let $e_\omega = \omega - \omega_h$ and let the filtered error $\tilde e_\omega$ be
defined as the solution to the problem
\begin{equation}\label{filtering}
-\delta^2 \Delta \tilde e_\omega + \tilde e_\omega = e_\omega.
\end{equation}
We introduce a norm on $\tilde e_\omega$ such that
$\tnorm{\tilde e_\omega}_\delta^2:=
\|\delta \nabla \tilde e_\omega\|^2 +\| \tilde e_\omega\|^2 = (e_\omega,\tilde
e_\omega)_\Omega$. This norm coincides with the $L^2$-norm for $\delta = 0$ and is
 related to the $H^{-1}$-norm for $\delta = 1$. By choosing
$\delta := \delta(h)$, i.e. by reducing the filter width with the mesh
size, we obtain a family of norms that become stronger as the mesh
size is reduced.

Using the above norm and the relations \eqref{NSpert1},
\eqref{eq:dual1}-\eqref{eq:dual2} as  well as the stability result of
Proposition \ref{prop:dualstability} we may derive a
posteriori estimates for the filtered quantity $\tilde e_\omega$. We here
derive the result for the abstract finite element element method
\eqref{NSFEM1}-\eqref{NSFEM2} and then show how these estimates can be
transformed into a priori error estimates, depending on the properties
of the stabilization operator $s(u_h,\omega_h;v_h)$. The use of weak
norms and stabilized finite element methods in the following estimates
draws on ideas from \cite{HMSW99} and \cite{Bu09, Bu12}.
\begin{theorem}(A posteriori error
  estimates)\label{est:a_posteriori}
\begin{align}
\tnorm{\tilde \omega - \tilde \omega_h}_\delta &\lesssim e^{\frac{T}{\tau_F}} \left(\frac{h}{\delta^2}
\right)^{\frac12} 
\sum_{i=0}^5 \mathcal{R}_i,\label{vorticity_apost}
\end{align}
with
\[
\mathcal{R}_0:= \|(\omega - \omega_h)(\cdot,0)\|,
\]
\[
\mathcal{R}_1 :=\int_0^T
\inf_{v_h \in V_h}\|h^{\frac12} (u_h \SCAL
  \nabla \omega_h - v_h)\|~\mbox{d}t,
\]
\[
\mathcal{R}_2:= \min(h,\nu^{\frac12}T^{\frac12}) \|\nu^{\frac12} \jump{\normal_F \SCAL
  \nabla \omega_h}\|_{\F\times I},
\]
\[
\mathcal{R}_3  :=  \int_0^T \|\omega_h(\cdot,t)\|_{L^\infty(\Omega)}
~\mbox{d}t  \min(c_0 \sup_{t \in I} \|\Psi_h(\cdot,t)\|_{\Delta,0},
c_1 h^{\frac12}
\sup_{t \in I} \|\omega_h(\cdot,t)\| )\\
\]
where 
\[
\|\Psi_h(\cdot,t)\|_{\Delta,s} := 
\|h^{s} \jump{n_F \cdot \nabla \Psi_h(\cdot,t)}\|_{\mathcal{F}}  + \inf_{v_h \in V_h^l}\left(\sum_{K\in
  \mathcal{T}_h} \|h^{\frac12+s}(\Delta
\Psi_h(\cdot,t) - v_h)\|^2_K\right)^{\frac12},
\]
\[
\mathcal{R}_4 := h^{\frac32} \int_0^T \|\partial_t \nabla
\omega_h\| ~\mbox{d}t
\]
and
\[
\mathcal{R}_5:=(U_0+\|u_h\|_{L^\infty(Q)})
\int_0^T s(u_h;\omega_h,\omega_h)^{\frac12} ~\mbox{d}t.
\]
The term $\mathcal{R}_4$ is omitted if the consistent mass matrix is
used. For the
velocities we have the estimate, for all $t \in I$,
\begin{equation}\label{velocity_apost}
\|(u - u_h)(\cdot,t)\| \lesssim \left(\|\Psi_h(\cdot,t)\|_{\Delta,\frac12}+ \tnorm{(\tilde \omega -
  \tilde \omega_h)(\cdot,t)}_1\right)
\end{equation}
where the second term in the right hand side may be a posteriori
bounded by taking $\delta=1$ in \eqref{vorticity_apost}.
\end{theorem}
\begin{proof}
By the definition of $\tilde e_\omega$ we have, taking $\xi_0 = \tilde
e_\omega$ in \eqref{eq:dual3},
\begin{align*}
\tnorm{\tilde e_\omega}_\delta^2 &= (e_\omega(T),\varphi_1(T))_\Omega + (e_\omega,-\partial_t \varphi_1- u \SCAL \nabla
\varphi_1- \nu \Delta \varphi_1)_{Q} \\ & \qquad + (e_{\Psi},-\Delta
\varphi_2-\nabla \omega_h \cdot \mbox{rot }\varphi_1)_{Q}\\
&=(e_\omega(0),\varphi_1(0))_\Omega + (\partial_t e_\omega+u\SCAL\nabla e_\omega+\mbox{rot } e_{\Psi} \SCAL
\nabla \omega_h,\varphi_1)_\Omega + (\nu \nabla e_\omega,\nabla \varphi_1)_\Omega \\ &
\qquad +
(\nabla e_{\Psi},\nabla \varphi_2)_\Omega - (e_\omega,\varphi_2)_\Omega.
\end{align*}
Using now the consistency relation \eqref{NSpert1} we obtain
\begin{align*}
\tnorm{\tilde e_\omega}_\delta^2 &= (e_\omega(0),(\varphi_1-\pi_L \varphi_1)(\cdot,0))_\Omega+ (\partial_t e+u\SCAL\nabla e+\mbox{rot } e_{\Psi} \SCAL
\nabla \omega_h,\varphi_1-\pi_L \varphi_1)_{Q} \\& + (\nu \nabla
e,\nabla (\varphi_1-\pi_L \varphi_1))_{Q} +
(\nabla e_{\Psi},\nabla (\varphi_2-\Pi \varphi_2) )_{Q} -
(e,\varphi_2-\Pi \varphi_2)_{Q}  \\ &
\qquad - (\partial_t
\omega_h,\pi_L \varphi_1)_{M,Q} + (\partial_t
\omega_h,\pi_L \varphi_1)_{Q} - s(u_h,\omega_h;\pi_L \varphi_1),
\end{align*}
where $\Pi:H^1(\Omega) \mapsto V_h^l$ will be taken as either $\pi_L$ or $\pi_V$.
Using the equations \eqref{NSweak1}-\eqref{NSweak2} and the definitions of the projections $\pi_L$ and $\pi_V$
we deduce for $\Pi:= \pi_V$,
\begin{align*}
\tnorm{\tilde e_\omega}_\delta^2 &= (e_\omega(0),(\varphi_1-\pi_L
\varphi_1)(\cdot,0))_\Omega-(u_h \SCAL \nabla \omega_h -v_h, \varphi_1-\pi_L
\varphi_1)_{Q}\\
& \qquad- (\nu \nabla \omega_h, \nabla (\varphi_1 -\pi_L
\varphi_1)_{Q} + (\omega_h, \varphi_2 - \pi_V \varphi_2)_{Q}\\
& \qquad \qquad- (\partial_t
\omega_h,\pi_L \varphi_1)_{M,Q} + (\partial_t
\omega_h,\pi_L \varphi_1)_{Q} - \int_0^T s(u_h,\omega_h;\pi_L \varphi_1)~\mbox{d}t,
\end{align*}
and similarly for $\Pi:=\pi_L$,
\begin{align*}
\tnorm{\tilde e_\omega}_\delta^2 &= (e_\omega(\cdot,0),(\varphi_1-\pi_L
\varphi_1)(\cdot,0))_\Omega-(u_h \SCAL \nabla \omega_h -v_h), \varphi_1-\pi_L
\varphi_1)_{Q}\\
& \qquad- (\nu \nabla \omega_h, \nabla (\varphi_1 -\pi_L
\varphi_1)_{Q} - (\nabla \Psi_h, \nabla(\varphi_2 - \pi_L \varphi_2))_{Q}\\
& \qquad \qquad- (\partial_t
\omega_h,\pi_L \varphi_1)_{M,Q} + (\partial_t
\omega_h,\pi_L \varphi_1)_{Q} - \int_0^T s(u_h,\omega_h;\pi_L \varphi_1)~\mbox{d}t.
\end{align*}
After some standard manipulation including integrations by parts, Cauchy-Schwarz
inequalities, trace inequalities the approximation results \eqref{approx1} and
\eqref{approx2} we may conclude, for $\Pi:=\pi_V$,
\begin{align*}
\tnorm{\tilde e_\omega}_\delta^2 & \lesssim \left(\frac{h}{\delta^2}
\right)^{\frac12} \Bigl(\|e_\omega(\cdot,0)\| + \int_0^T \inf_{v_h \in V_h}
\|h^{\frac12}(u_h \SCAL
  \nabla \omega_h - v_h)\| ~\mbox{d}t \\+ \min(h,&\nu^{\frac12} T^{\frac12}) \|\nu^{\frac12} \jump{\normal_F \SCAL
  \nabla \omega_h}\|_{\F\times I} + c_1 h^{\frac12} \sup_{t\in I} \|\omega_h(\cdot,t)\| \int_0^T \|\omega_h(\cdot,t)\|_{L^\infty(\Omega)} ~\mbox{d}t
\\&+ h^{\frac32} \int_0^T \|\partial_t \nabla
\omega_h\| ~\mbox{d}t +(U_0+ \|u_h\|_{L^\infty(Q)})
\int_0^T s(u_h;\omega_h,\omega_h)^{\frac12} ~\mbox{d}t\Bigr)\\
& \qquad \times (\sup_{t\in I} \|\delta \nabla
\varphi_1(\cdot,t)\| + \|\delta \nu^\frac12 D^2 \varphi_1\|_{Q}).
\end{align*}
If $\Pi:=\pi_L$ the fourth term on the right hand side is replaced using 
$$
 (\nabla \Psi_h, \nabla(\varphi_2 - \pi_L \varphi_2))_{Q}  \lesssim
 \left( \frac{h}{\delta^2} \right)^{\frac12} 
 c_0 \sup_{t \in I} \|\Psi_h(t)\|_{\Delta,0} \int_0^T \|\delta\nabla
\varphi_2(\cdot,t)\| ~\mbox{d}t,
$$
followed by the bound \eqref{phi_2stab} on $\varphi_2$.
The estimate \eqref{vorticity_apost} now follows by taking the minimum
of the two expressions and noting that by \eqref{phi_1stab} 
\[
\sup_{t\in I} \|\delta \nabla
\varphi_1(\cdot,t)\| + \|\delta \nu^\frac12 D^2 \varphi_1\|_{Q}
\lesssim e^{\frac{T}{\tau_F}} \tnorm{\tilde e_\omega}_\delta.
\]
The velocity estimate \eqref{velocity_apost} is obtained by noting
that, with $e_\Psi := \Psi - \Psi_h$,
\begin{equation*}
\|u - u_h\|^2 := \|\nabla e_\Psi \|^2 = (\nabla e_\Psi, \nabla
(e_\Psi - \pi_L e_\Psi)) + (e_\omega, \pi_L e_\Psi).
\end{equation*}
Using the equation \eqref{NSweak2} we have
\begin{multline*}
\|\nabla e_\Psi \|^2 = (\nabla \Psi_h, \nabla
(e_\Psi - \pi_L e_\Psi)) + (\omega,e_\Psi) - (\omega_h, \pi_L e_\Psi)
\\ =  (\nabla \Psi_h, \nabla
(e_\Psi - \pi_L e_\Psi)) + (\omega - \omega_h,e_\Psi) 
\end{multline*}
Let $\tilde e$ be the solution of \eqref{filtering} with
$\delta=1$.
Then
\[
\|u - u_h\|^2 = (\nabla \Psi_h, \nabla
(e_\Psi - \pi_L e_\Psi))_\Omega + (\nabla \tilde e_\omega, \nabla e_\Psi)_\Omega +
(\tilde e_\omega,  e_\Psi)_\Omega.
\]
By an integration by parts in the first term, followed by a Cauchy-Schwarz
inequality and the Poincar\'e-Friedrichs inequality in the last term
we may write
\begin{multline*}
\|u - u_h\|^2 \lesssim \|h_F^{\frac12} \jump{\nabla \Psi_h}\|_{\F}
\|h^{-\frac12}(e_\Psi - \pi_L e_\Psi)\|_{\F} \\[3mm]+ \left(\sum_{K
    \in \mathcal{T}_h} \|h(\Delta \Psi_h -
v_h)\|_K^2\right)^{\frac12} \|h^{-1}(e_\Psi - \pi_L
e_\Psi)\|\\[3mm] 
+ \tnorm{(\tilde \omega -
  \tilde \omega_h) }_1 \|(u - u_h)\|.
\end{multline*}
By elementwise trace inequalities and the approximation property \eqref{approx1} we have
\[
\|h^{-\frac12}(e_\Psi - \pi_L e_\Psi)\|_{\F} + \|h^{-1}(e_\Psi - \pi_L e_\Psi)\|\lesssim \|u -
u_h\|
\]
by which we conclude.
\end{proof}\\
If the stability properties of the stabilized method are sufficient,
these a posteriori error estimates translate into a priori error
estimates. We propose two strategies for this. One using stability concepts based on Sobolev injections for
discrete spaces and one based on
monotonicity, applicable to monotone stabilized finite element methods
and monotone implicit large eddy methods. The advantage of the former is
that it allows the derivation of a
priori error estimates for quasi linear terms
$s(u_h;\omega_h,v_h)$ and the use of the consistent mass matrix. The
latter technique on the other hand allows for the derivation of a
priori error estimates with precise control of the constants in the
estimates. We will use the notion of the discrete maximum principle (DMP)
and the associated, DMP-property of the forms defining a finite
element method introduced in \cite{BE05}.
\begin{proposition}\label{Sobapriori}
Assume that the mass $(\cdot,\cdot)_M$ is evaluated exactly and that
in addition to \eqref{stab_bound1} and \eqref{stab_bound2}
the following stability estimate holds for all $t>0$,
\begin{equation}\label{Sobolevinj}
\|\omega_h\|_{L^\infty(\Omega)} \lesssim c(h) (\|\omega_h\| + s(u_h,\omega_h,\omega_h)^{\frac12}).
\end{equation}
Then there holds for all $\epsilon>0$,
\[
\tnorm{(\tilde \omega - \tilde \omega_h)(T)}_\delta \lesssim
e^{\frac{T}{\tau_F}} \left(\frac{h}{\delta^2}
\right)^{\frac12} (c_\epsilon h^{-\epsilon} + c(h) h^{\frac12})
\]
and
\[
\|(u-u_h)(\cdot,T)\| \lesssim \inf_{v_h \in W_h^{l-1}} \|(u -
v_h)(\cdot,T)\|  + e^{\frac{T}{\tau_F}} h^{\frac12} (c_\epsilon h^{-\epsilon} + c(h) h^{\frac12}).
\]
\end{proposition}
\begin{proof}
First we recall that $\|\omega_h(\cdot,0)\| \leq \|\omega(\cdot,0)\|$. Then
by \eqref{stab_bound1} and \eqref{enstrophy_stability}
\[
\int_0^T
\inf_{v_h \in V_h}\|h^{\frac12} (u_h \SCAL
  \nabla \omega_h - v_h)\|~\mbox{d}t \lesssim T^{\frac12} \Bigl(\int_0^T
  s(u_h;\omega_h,\omega_h) ~\mbox{d}t\Bigr)^{\frac12} \lesssim T^{\frac12} \|\omega(\cdot,0)\|.
\]
Using an elementwise trace inequality and \eqref{enstrophy_stability} we also have 
\[
\min(h,\nu^{\frac12}T^{\frac12}) \|\nu^{\frac12} \jump{\normal_F \SCAL  \nabla \omega_h}\|_{\F\times I}
\leq h^{\frac12} \|\nu^{\frac12} \nabla \omega_h\|_{Q}  \lesssim h^{\frac12} \|\omega(\cdot,0)\|.
\]
For $\mathcal{R}_3$ we use the discrete Sobolev injection
\eqref{Sobolevinj}
to deduce
\begin{multline*}
h^{\frac12} \sup_{t \in I} \|\omega_h(\cdot,t)\| \int_0^T
\|\omega_h(\cdot,t)\|_{L^\infty(\Omega)} ~\mbox{d}t \\ \lesssim h^{\frac12}
\|\omega_h(\cdot,0)\| c(h) \int_0^T(\|\omega_h\|+
s(u_h,\omega_h,\omega_h)^{\frac12}) ~\mbox{d}t\\
 \lesssim  h^{\frac12} c(h)  \|\omega_h(\cdot,0)\|^2.
\end{multline*}
The only remaining term is the stabilization term, which is not
innocent since we do not have an a priori bound on the
factor $\|u_h\|_{L^\infty(Q)}$. Here we use \eqref{ubound_infty} to
deduce, for all $t>0$ and $q>2$,
\[
\|u_h\|_{L^\infty(\Omega)}
\int_0^T s(u_h;\omega_h,\omega_h)^{\frac12} ~\mbox{d} t \leq c_q
\|\omega_h\|_{L^q(\Omega)} T^{\frac12} \Bigl( \int_0^T s(u_h;\omega_h,\omega_h) ~\mbox{d} t\Bigr)^{\frac12}
\]
and by a global inverse inequality and the bound
\eqref{enstrophy_stability} we may conclude
\[
\|u_h\|_{L^\infty(Q)}
\int_0^T s(u_h;\omega_h,\omega_h)^{\frac12} ~\mbox{d} t \lesssim c_q
h^{\frac{2-q}{q}}  T^{\frac12} \sup_{t \in I}
\|\omega_h(\cdot,t)\|\|\omega_h(\cdot,0)\|
\]
and the estimate follows taking $\epsilon = (q-2)/q$. Note that the
constant $c_q$ explodes as $q \rightarrow 2$.

The bound in the $L^2$-norm for the velocities follows as before from the
vorticity estimate using, with $C_h$ denoting the Cl\'ement interpolant,
\begin{align*}
\|u - u_h\|^2 & = \|\nabla e_\Psi \|^2 = (\nabla  e_\Psi , \nabla (\Psi -
C_h \Psi)) - (e_\omega, \Psi_h - C_h \Psi) \\[3mm]
& \leq \|u - u_h\| \|\nabla (\Psi -C_h \Psi)\| - (-\Delta \tilde
e_\omega  + \tilde
e_\omega, (C_h \Psi - \Psi_h)) \\[3mm]
&
\leq \|u - u_h\| \|\nabla (\Psi -C_h \Psi)\| + \tnorm{\tilde
\omega - \tilde \omega_h}_1 \|C_h \Psi - \Psi_h\|_{H^1(\Omega)}.
\end{align*}
We conclude by using the $H^1$-stability of the Cl\'ement interpolant, a Poincar\'e inequality and finally by dividing both sides with $\|u - u_h\|$.
\end{proof}\\
\begin{proposition}\label{DMPapriori}(A priori error estimate using monotonicity)
Assume that $Re_h>1$, that the mass $(\cdot,\cdot)_M$ is evaluated
using nodal quadrature, that the form $a(u_h;\omega_h,v_h)+s(u_h;\omega_h,v_h)$
has the DMP property as defined in \cite{BE05} and that
\eqref{stab_bound1}-\eqref{stab_bound2} are satisfied as well as the
assumptions of Lemma \ref{FEMstability}
Then there holds
\[
\tnorm{(\tilde \omega - \tilde \omega_h)(T)}_\delta \lesssim
e^{\frac{T}{\tau_F}} \left( \frac{h}{\delta^2} \right)^{\frac12}
\]
and
\[
\|(u-u_h)(\cdot,T)\| \lesssim \inf_{v_h \in  W^{l-1}_h} \|(u -
v_h)(\cdot,T)\| +  e^{\frac{T}{\tau_F}} h^{\frac12}.
\]
\end{proposition}
\begin{proof}
The terms $\mathcal{R}_0-\mathcal{R}_2$ are bounded as in the  proof of Proposition
\ref{Sobapriori}.
Since by assumption the spatial discretization of \eqref{NSFEM1} has the
DMP property and the mass-matrix is evaluated using nodal quadrature, we know from \cite{BE03,BE05} that 
\[
\|\omega_h\|_{L^\infty(Q)} = \|\omega_h(\cdot,0)\|_{L^\infty(\Omega)}.
\]
Hence by \eqref{ubound_infty} $\|u_h\|_{L^\infty(Q)} \leq c_\infty
\|\omega_h(\cdot,0)\|_{L^\infty(\Omega)}$. We may then use these
$L^\infty$-bounds together with the stabilities of Lemma
\ref{FEMstability} to upper bound the remaining residual quantities of
\eqref{vorticity_apost}.
Using \eqref{ubound_infty} and \eqref{enstrophy_stability} we
immediately have
\[
h^{\frac12} \sup_{t \in I} \|\omega_h(\cdot,t)\| \int_0^T
\|\omega_h(\cdot,t)\|_{L^\infty(\Omega)} ~\mbox{d}t \lesssim h^{\frac12} T \|\omega_h(\cdot,0)\|\|\omega_h(\cdot,0)\|_{L^\infty(\Omega)}
\]
For the residual term resulting from the mass lumping we have using
the stability \eqref{timederstab}
\begin{multline*}
h^{\frac32} \int_0^T \|\partial_t \nabla \omega_h\| ~\mbox{d}t
\lesssim T^{\frac12} (U_0+\|u_h\|_{L^\infty(Q)})\Bigl(\int_0^T (s(u_h;\omega_h,\omega_h)
+ \|\nu^{\frac12} \nabla \omega_h\|^2)~\mbox{d}t\Bigr)^{\frac12} \\
\leq
T^{\frac12} U_0^{\frac12} \|\omega(\cdot,0)\|.
\end{multline*}
The remaining contribution from the stabilization is bounded as before
using the maximum principle and \eqref{enstrophy_stability}. The proof
of the 
$L^2$-norm estimate on the velocities is identical to that of
Proposition \ref{Sobapriori}
\end{proof}\\
Note that only the proof of Proposition \ref{DMPapriori} uses the
assumption $Re_h>1$ and only to control the non-consistent mass
term. This constraint is likely to vanish if the method is analysed
using techniques appropriate for parabolic problems, since mass lumping
is known to be stable for dominant diffusion (see for instance
\cite{Thom97}).
\section{Stabilized finite element methods}
The estimates of Theorem \ref{est:a_posteriori} holds for any finite
element method on the form \eqref{NSFEM1}-\eqref{NSFEM2}. Indeed by
taking $s(\cdot;\cdot,\cdot) \equiv 0$ the standard Galerkin method is
included. This means that in general the effect of stabilization can be observed
only a posteriori, by observing smaller residuals for the stabilized
formulations. In Propositions \ref{DMPapriori} and \ref{Sobapriori} we propose a priori
estimates derived from the a posteriori error estimates under special 
assumptions on the properties of the stabilizing terms.
These
can be proven to hold only for stabilized finite element methods, since the standard
Galerkin method does not allow for a control of the second term of the
right hand side of \eqref{vorticity_apost} independently of the
viscosity, nor can \eqref{stab_bound1} and \eqref{Sobolevinj} be
proven to hold. In this section we will suggest some stabilization
operators that satisfy the assumptions necessary for the results of
the abstract analysis to hold. We will consider the following cases:
\begin{enumerate}
\item linear artificial viscosity, in which the numerical viscosity is
  increased so that the mesh Reynolds number always is one. Using a
  lumped mass matrix together with
  anisotropic viscosity we may design the scheme to satisfy a discrete maximum
  principle, giving a priori control of
  $\|\omega_h\|_{L^\infty(Q)}$. When the consistent mass matrix is
  used one may proved that \eqref{Sobolevinj} holds giving once again
  a priori estimates,
at the price of a logarithmic factor.
\item high order stabilization, we propose to stabilize the
  jump of the streamline derivative. This scheme does not yield a
  maximum principle, so the residuals can not be completely a priori
  bounded. 
The scheme has some interesting conservation properties for
  two-dimensional Navier-Stokes' computations that we will point out.
If a nonlinear stabilization term is added and mass-lumping is used
 the solution may be made monotone and the a priori error estimate of
 Proposition \ref{DMPapriori} holds, this time with the possibility of higher
order convergence in the smooth portion of the flow. Finally if the
consistent mass matrix is used and stabilization is added also in the crosswind direction, an estimate of
the type \eqref{Sobolevinj} can be shown to hold leading to a priori
error bounds using Proposition \eqref{Sobapriori}.
\end{enumerate}
\subsection{Methods using consistent mass matrix}
We consider first the stabilization method obtained by penalizing the
jumps of the streamline derivative over element faces. We use the
exact mass matrix in \eqref{NSFEM1} and the stabilizing operator
\begin{equation}\label{CIP}
s_{sd}(u_h,\omega_h,v_h):= \gamma \sum_{F \in \F} U_0^{-1} ( h^2_F \jump{ u_h \SCAL \nabla \omega_h} ,\jump{ u_h \SCAL \nabla v_h})_{F}.
\end{equation}
For this formulation the following stability estimates hold
\begin{lemma}
\begin{equation}\label{enstrophy_stability_CIP}
\sup_{t \in I} \|\omega_h(\cdot,t)\|^2 + 2\|\nu^{\frac12} \nabla
\omega_h\|^2_{Q} + 2\gamma U_0^{-1} \|h_F \jump{u_h \SCAL \nabla \omega_h}\|_{\F}^2 \leq \|\omega_h(\cdot,0)\|^2
\end{equation}
and if the consistent mass matrix is used, 
\begin{equation}\label{energy_stability_CIP}
\|u_h(\cdot,T)\|^2 + 2\|\nu^{\frac12}
\omega_h\|^2_{Q} = \|u_h(\cdot,0)\|^2
\end{equation}
\end{lemma}
\begin{proof}
the proof of \eqref{enstrophy_stability_CIP} is an immediate
consequence of \eqref{enstrophy_stability} and the definition
\eqref{CIP}.
The inequality \eqref{energy_stability_CIP} follows by observing that
\[
s_{sd}(u_h,\omega_h,\Psi_h) =  \gamma \sum_{F \in \F} ( U_0^{-1} h^2_F \jump{ u_h \SCAL
  \nabla \omega_h} ,\jump{ \mbox{rot } \Psi_h \SCAL \nabla
  \Psi_h})_{F} = 0.
\]
\end{proof}\\
Observe that the method dissipates enstrophy but conserves energy
exactly as the physics of the problem suggests. Using known results on
interpolation between discrete spaces it is also straightforward to
show (see \cite{BH04}),
\[
\inf_{v_h \in V_h} \|h^{\frac12}(u_h \cdot \nabla \omega_h - v_h)\|^2 \lesssim s_{sd}(u_h,\omega_h,\omega_h).
\]
Unfortunately this
stabilization operator can not be shown to satisfy
\eqref{Sobolevinj}. For this we need the stabilization to act also in
the crosswind direction. We therefore propose the following two
stabilization operators, the first is the standard artificial
viscosity method 
\begin{equation}\label{s_av}
s_{av}(u_h;\omega_h,\omega_h):=(\gamma h (U_0 + |u_h|)^2 U_0^{-1} \nabla
\omega_h,\nabla v_h)
\end{equation}
and the second is a modification of \eqref{CIP} where
also the crosswind gradient is penalized 
defined by
\begin{equation} 
s_{cd}(u_h,\omega_h,v_h) := s_{sd}(u_h,\omega_h,v_h) +
\gamma_1 \sum_{K\in \mathcal{T}_h} U_0 h_K^\mu \int_{\partial K}
\jump{n_F \cdot\nabla \omega_h}\jump{n_F \cdot \nabla v_h }
~\mbox{d}s \label{s_cd}
\end{equation}
Observe that the first part of $s_{cd}$ ensures the
satisfaction of \eqref{stab_bound1} and as we shall see the second
part is necessary for \eqref{Sobolevinj} to hold.
\begin{proposition}
Both stabilization operators \eqref{s_av} and \eqref{s_cd} satisfy
\eqref{stab_bound1} and \eqref{stab_bound2}.
The stabilization operator $s_{av}(\cdot;\cdot,\cdot)$ satisfy
\eqref{Sobolevinj} with $c(h) \sim h^{-\frac12} (1+ |log(h)|)$ and $s_{cd}(\cdot;\cdot,\cdot)$
satisfy \eqref{Sobolevinj} with $c(h) \sim h^{-\tfrac{1+\mu}{4}}(1+
|log(h)|)$, $\mu>0$.
\end{proposition}
\begin{proof}
The proofs of \eqref{stab_bound1} - \eqref{stab_bound2} are
consequences of the Cauchy-Schwarz inequality and in the case of
$s_{cd}$ trace inequalities. To prove \eqref{Sobolevinj} we note that
in two space dimensions there holds (see \cite{Sco76}),
\[
\|\omega_h\|_{L^{\infty}(\Omega)} \lesssim (1+|\log(h)|) \|\omega_h\|_{H^1(\Omega)}.
\]
This allows us to conclude for $s_{av}$. For $s_{cd}$ we use that
\[
\|\nabla \omega_h\| \leq \Bigl(\sum_{F\in\F} \int_{F}
|\jump{\nabla \omega_h \cdot n_{F}}| |\omega_h|
~\mbox{d}s \Bigr)^{\frac12}.
\]
A Cauchy-Schwarz inequality followed by a trace inequality in the
right hand side leads to
\[
\|\nabla \omega_h\| \lesssim \Bigl(\sum_{K \in \mathcal{T}_h} h^{-\frac{1+\mu}{2}}
\|\omega_h\|_K \|h^{\frac{\mu}{2}} \jump{n_f \cdot \nabla
  \omega_h}\|_{\partial K}\Bigr)^{\frac12} 
\lesssim h^{-\frac{1+\mu}{4}} (\|\omega_h\| + s_{cd}(u_h;\omega_h,\omega_h)^{\frac12}).
\]
\end{proof}\\
Since the assumptions of Proposition \ref{Sobapriori} are satisfied, we may
conclude that the method \eqref{NSFEM1}-\eqref{NSFEM2} using the
stabilization \eqref{s_av} statisfy the a priori error bounds for $\epsilon>0$
\[
\tnorm{(\tilde \omega - \tilde \omega_h)(T)}_\delta \lesssim
e^{\frac{T}{\tau_F}} \left(\frac{h}{\delta^2}
\right)^{\frac12} (c_\epsilon h^{-\epsilon} + 1 + |\log(h)|)
\]
and
\[
\|(u-u_h)(\cdot,T)\| \lesssim \inf_{v_h \in  W^{l-1}_h} \|(u -
v_h)(\cdot,T)\|  + e^{\frac{T}{\tau_F}} \left(\frac{h}{\delta^2}
\right)^{\frac12} (c_\epsilon h^{-\epsilon} + 1 + |\log(h)|).
\]
Similarly we have the following estimates if the stabilization
\eqref{s_cd} is used.
\[
\tnorm{(\tilde \omega - \tilde \omega_h)(T)}_\delta \lesssim
e^{\frac{T}{\tau_F}} \left(\frac{h}{\delta^2}
\right)^{\frac12} (c_\epsilon h^{-\epsilon} + (1 + |\log(h)|) h^{-\frac{\mu-1}{4}})
\]
and
\[
\|(u-u_h)(\cdot,T)\| \lesssim \inf_{v_h \in  W^{l-1}_h} \|(u -
v_h)(\cdot,T)\|  + e^{\frac{T}{\tau_F}} h^{\frac12} (c_\epsilon h^{-\epsilon} + (1 + |\log(h)|) h^{-\frac{\mu-1}{4}}).
\]
We see that if we take $\mu=1$ in \eqref{s_cd} we get the same order
for the two methods, however if we want the method to have optimal
convergence for smooth solutions we choose $\mu=2$ and $l=2$, resulting in an
a priori convergence order of $O(h^{\frac14})$ in the non-smooth case.
\subsection{Monotone methods}
Since the consistent mass matrix is non-monotone we herein only
consider methods using lumped mass. Monotone methods can also be designed
using a nonlinear switch that changes the local quadrature as a
function of the solution $\omega_h$ so that the consistent mass is
used away from local extrema to reduce the dispersion error known to
haunt mass-lumping schemes, such methods are beyond the
scope of the present paper.
\subsubsection{Linear artificial viscosity}
A monotone method using linear artificial viscosity is obtained by
taking (see \cite{BE05})
\begin{equation}\label{s:art_visc}
s(u_h,\omega_h,v_h):= \gamma \sum_K  (\max(U_0,\|u_h\|_{L^\infty(K)})
h^2_K \sum_{F \in \partial K} (\nabla \omega_h
\times \normal_F ,\nabla v_h \times \normal_F)_{F}.
\end{equation}
Then the estimates
\eqref{enstrophy_stability}-\eqref{energy_stability} hold
and we observe that there exists positive constants $c_1,c_2$ such that
\[
c_1 \||u_h|^\frac12 h^\frac12 \nabla \omega_h\|^2_{\Omega} \leq
s(u_h,\omega_h,\omega_h) \leq c_2 \||u_h|^\frac12 h^\frac12 \nabla \omega_h\|^2_{\Omega}.
\]
Let the mass matrix be evaluated using nodal quadrature so that the
matrix corresponding to $(\cdot,\cdot)_M$ is diagonal. We may use the
theory of \cite{BE03, BE05} to prove that the operator
$a(\omega_h,v_h)+s(u_h,\omega_h,v_h)$ has the DMP-property and hence the following
discrete maximum principle holds
\[
\|\omega_h\|_{L^\infty(Q)} = \|\omega_h(\cdot,0)\|_{L^\infty(\Omega)}.
\] 
This requires the parameter $\gamma$ to be chosen large enough,
however it does not require any additional acute condition on the
mesh, since the discretization of the Laplace operator results in an M-matrix on Delaunay meshes.
Since by the maximum principle, $\|u_h\|_{L^\infty(Q)} \lesssim
\|\omega_h(\cdot,0)\|_{L^\infty(\Omega)}$ we have
\begin{equation}\label{grad_bound}
\||u_h| h^\frac12 \nabla \omega_h\|^2_{Q} \lesssim \|u_h\|^{\frac12}_{L^\infty(Q)} 
\int_0^Ts(u_h,\omega_h,\omega_h)~\mbox{d}t \lesssim \|\omega_h(\cdot,0)\|_M^2
\end{equation}
which proves \eqref{stab_bound1} with $v_h = 0$. It is straightforward
to prove also \eqref{stab_bound2}.
Comparing with Proposition \ref{DMPapriori} we conclude that the
assumptions are satisfied and hence that the Proposition holds for
\eqref{NSFEM1}-\eqref{NSFEM2} with stabilization given by
\eqref{s:art_visc} and the mass matrix evaluated using nodal
quadrature.
\subsubsection{Nonlinear artificial viscosity}
Here we assume that $l=1$ so that both $\omega_h$ and $\Psi_h$ are
discretized using piecewise affine elements. We propose a stabilization term consisting of one linear part and
one nonlinear part. The role of the nonlinear part is to ensure that
the form $a(\cdot;\cdot,\cdot)+s(\cdot;\cdot,\cdot)$  has the DMP
property. The linear part is necessary to ensure that
the inequality \eqref{stab_bound1} holds. We define
\begin{align}\label{nonlin1}
s(u_h;\omega_h,v_h) & := s_{sd}(u_h;\omega_h,v_h) \\
&+
\gamma_2 \sum_K h^2
 \sum_{F \in \partial K}  R_F(u_h,\omega_h) (\mbox{sign}(\nabla \omega_h
\times \normal_F ),\nabla v_h \times \normal_F)_{F}
\end{align}
where 
\[
R_F(u_h,\omega_h) := \|u_h\|_{L^\infty(\Delta_F)} (1+U_0^{-1} \|u_h\|_{L^\infty(\Delta_F)})m_F(\jump{n_F
  \cdot \nabla \omega_h}) 
\]
with $\Delta_F := \cup_{K \in \mathcal{T}_h; K\cap F \ne  \emptyset}
K$ and
\[
m_F(\jump{n_F
  \cdot \nabla \omega_h}) = \max_{\substack{ F' \in \F \\ F'
    \in \partial K'; K' \cap F = F}} \|\jump{n_F
  \cdot \nabla \omega_h}\|_{F'}.
\]
 It
is shown in \cite{BE05} that with this definition
$a(u_h;\omega_h,v_h) + s(u_h;\omega_h,v_h)$ has the DMP-property for
$\gamma_2$ large enough. 
Since the bounds \eqref{stab_bound1}-\eqref{timederstab} also hold, 
the assumptions of Proposition \ref{DMPapriori} are
satisfied and its estimates hold.
We conclude that for the methods defined by mass lumping and the
stabilization operators \eqref{s:art_visc} or \eqref{nonlin1} the
following estimates hold
\[
\tnorm{(\tilde \omega - \tilde \omega_h)(T)}_\delta \lesssim
e^{\frac{T}{\tau_F}} \left(\frac{h}{\delta^2}
\right)^{\frac12}
\]
and
\[
\|(u-u_h)(\cdot,T)\| \lesssim \inf_{v_h \in  W^{l-1}_h} \|(u -
v_h)(\cdot,T)\|  + e^{\frac{T}{\tau_F}} \left(\frac{h}{\delta^2}
\right)^{\frac12}.
\]
\section{Conclusion}
We have shown that under a certain structural assumption on the
solution of the two dimensional Navier-Stokes' equation one may derive
robust error estimates with an order in $h$, independent of both
the Reynolds number and high order Sobolev norms of the exact
solution. Robustness is obtained for a class of stabilized finite
element methods. The estimates are
both on a posteriori form, and on
a priori form, providing an upper bound on the error. Due to the
strong assumptions on the mesh the present a posteriori error
estimates are not immediately suitable for use in adaptive algorithms,
but a more detailed analysis may allow the mesh assumptions to be
relaxed.
If the solution
is smooth we also prove that optimal convergence may be obtained,
provided the stabilization operator is weakly consistent to the right order.

Observe that it is natural that the LES estimate has much poorer
convergence order, since we may assume no smoothness of the exact
solution. Even the large scales are assumed to have moderate gradients
only.

We show how several stabilized methods enter the framework, both first
and second order accurate ones. The interest of the first order
artificial viscosity method is primarily its close relationship to the 
vertex centered finite volume method. Note also that the estimates
with an order proposed herein for nonlinear monotone schemes to the
best of our knowledge are the first of their kind in the literature.

It appears that for implicit large eddy simulations 
both the
estimate \eqref{classic_estimate} for smooth solutions and the
estimate \eqref{ILES_estimate} for rough solutions derived herein are
desirable properties for the theoretical justification of a method.

Future work will focus on
numerical investigations both in two and three space dimensions. Of
particular interest is to study the stability of the incompressible
Euler equations to see if the limit estimate with no allowed small
scales is sharp.
\section*{Acknowledgment} Partial funding for this research was
provided by EPSRC
(Award number EP/J002313/1).
\appendix
\section{Proof of Proposition \ref{smooth_apriori}}
We introduce the discrete errors, with $I_h$ denoting
the Lagrange interpolant
\[
e_{h,\psi}:= \Psi_h - I_h \Psi \mbox{ and } e_{h,\omega}:= \omega_h -
\pi_L \omega.
\]
First consider the second equation \eqref{NSFEM2} and use Galerkin
orthogonality
\[
\|\nabla e_{h,\Psi}\|^2 = (\nabla (\Psi - I_h \Psi),\nabla e_{h,\Psi}) - (\omega - \omega_h,e_{h,\Psi}).
\]
Applying Poincar\'es inequality followed by Cauchy Schwarz inequality
we obtain the following bound for $\Psi$ in terms of the error in the vorticity
\[
\|\nabla e_{h,\Psi}\| \lesssim \|\nabla (\Psi - I_h \Psi)\| + \|\omega -
\pi_L \omega_h\| + \|e_{h,\omega}\|.
\]
Consider now the equation \eqref{NSFEM1} taking $v_h = e_{h,\omega}$
and observing that there holds
\begin{multline*}
\frac12 \frac{\mathrm{d}}{\mathrm{dt}} \|e_{h,\omega}\|^2 + s(u_h;e_{h,\omega},e_{h,\omega}) \\
= (\partial_t (\omega - \pi_L \omega), e_{h,\omega}) + (\omega, u \cdot \nabla e_{h,\omega})\\ - (\pi_L \omega, u_h \cdot \nabla e_{h,\omega}) - s(u_h;\pi_L \omega,e_{h,\omega}).
\end{multline*}
By integration by parts in time we see that the first term on the right hand 
side is zero, by the orthogonality of the $L^2$-projection. We then add and subtract $u_h$ in the second term on the right hand side to obtain
\begin{multline*}
\frac12 \frac{\mathrm{d}}{\mathrm{dt}} \|e_{h,\omega}\|^2 +s(u_h;e_{h,\omega},e_{h,\omega}) 
=  (\omega, (u - u_h) \cdot \nabla e_{h,\omega})\\ + (\omega - \pi_L \omega, u_h \cdot \nabla e_{h,\omega}) - s(u_h;\pi_L \omega,e_{h,\omega}) = I + II + III.
\end{multline*}
In the first term on the right hand side we now reintegrate by parts and use 
Cauchy-Schwarz inequality,
\[\begin{aligned}
I & \leq \|\omega\|_{W^{1,\infty}} \|u - u_h\| \|e_{h,\omega}\| \\
  & \lesssim \|\omega\|_{W^{1,\infty}} (\|\nabla (\Psi - I_h \Psi)\| + \|\omega - \omega_h\|)\|e_{h,\omega}\| \\
& \lesssim  \|\omega\|_{W^{1,\infty}} (\|\nabla (\Psi - I_h \Psi)\|^2+  \|\omega - \pi_L \omega\|^2 + \|e_{h,\omega}\|^2).
\end{aligned}
\]
In the second term we use the orthogonality of the
$L^2$-projection to retract some function $v_h$ and then apply \eqref{stab_bound1},
\[\begin{aligned}
II & =  (\omega - \pi_L \omega, u_h \cdot \nabla e_{h,\omega} - v_h)\\
& \leq c \|h^{-\frac12} (\omega - \pi_L \omega)\| s(u_h;e_{h,\omega},e_{h,\omega})^\frac12\\
& \leq  c h^{2s-1} \|\omega\|_{H^s}^2 +  \frac14 s(u_h;e_{h,\omega},e_{h,\omega}).
\end{aligned}
\]
For the stabilization term finally we apply the Cauchy-Schwarz
inequality and an arithmetic-geometric inequality to obtain
\begin{equation}
III = s(u_h;\pi_L \omega,e_{h,\omega})  \leq s(u_h;\pi_L \omega,\pi_L \omega) + 
\frac14 s(u_h;e_{h,\omega},e_{h,\omega}).
\end{equation}
Then we observe that by adding and subtracting $I_h \mbox{rot }
  \Psi$ we may write
\[
s(u_h;\pi_L \omega,\pi_L \omega) \lesssim s(u_h-I_h \mbox{rot }
  \Psi;\pi_L \omega,\pi_L \omega) + s(I_h \mbox{rot }
  \Psi;\pi_L \omega,\pi_L \omega)
\]
and using the definition \eqref{CIP} and the stability of the
$L^2$-projection on quasi uniform meshes, we have, 
\begin{align*}
s(u_h-I_h \mbox{rot }
  \Psi;\pi_L \omega,\pi_L \omega) & \lesssim \sum_{F \in \F} \int_F h^2 |u_h-I_h \mbox{rot }
  \Psi|^2 |\nabla \pi_L \omega_h|^2 ~\mbox{d}s  \\ &\lesssim \|\nabla
  \omega\|_{L^\infty(\Omega)}^2 h \|u_h-I_h \mbox{rot }
  \Psi\|^2\\
\lesssim \|\nabla
  \omega\|_{L^\infty(\Omega)}^2 h (\|\nabla \Psi - I_h \nabla \Psi)\|^2& + \|\nabla (\Psi - I_h \Psi)\|^2  +
  \|\omega - \pi_L \omega\|^2 + \|e_{h,\omega}\|^2)
\end{align*}
and then
\[
\begin{aligned}
s(I_h \mbox{rot }
  \Psi& ; \pi_L \omega, \pi_L \omega)\\
&\leq \|I_h \mbox{rot }
  \Psi\|^2_{L^\infty(Q)} \sum_K h_K \left(\|\nabla (\omega - \pi_L \omega)\|^2_K+h_K^2\|\nabla (\omega - \pi_L \omega)\|^2\right)\\
&\leq \|u\|^2_{L^\infty(Q)} C h_K^{2s-1} \| \omega \|_{L^2(I;H^s(\Omega))}^2.
\end{aligned}
\]
We conclude by collecting the upper bounds for the terms $I-III$,
applying approximability and Gronwall's lemma that
\begin{align*}
\|e_{h,\omega}(\cdot,T)\|^2 & +\int_0^T s(u_h;e_{h,\omega},e_{h,\omega})
\lesssim \exp(c T \| \omega\|_{W^{1,\infty}(Q)} ) \|
\omega\|_{W^{1,\infty}(Q)} \\ & \times  (h^{2l} \|\Psi\|^2_{L^2(I;H^{l+1}(\Omega)}
+ h^{2k+1} \|\omega\|^2_{L^2(I;H^{2}(\Omega))}).
\end{align*}
Here we assumed $ h \| \omega\|_{W^{1,\infty}(Q)} \lesssim 1$ and neglected
the dependence of $ \|u\|^2_{L^\infty(Q)}$ (that is upper bounded by $\|\omega(\cdot,0)\|^2_{L^\infty(\Omega)}$).
It follows that for $l=2$ and sufficiently smooth solutions we have
\[
\|(\omega - \omega_h)(\cdot,T)\| + \|(u - u_h)(\cdot,T)\| \lesssim h^{\frac32}.
\]

\bibliographystyle{plain}
\bibliography{Biblio}
\end{document}